\documentclass{amsart}
\usepackage{geometry}
\usepackage[utf8]{inputenc}
\usepackage{amsmath,amsfonts}
\usepackage{amsthm}
\usepackage{xcolor}
\usepackage{algorithm,algpseudocode}
\usepackage{bbm}
\usepackage[mathscr]{euscript}
\usepackage{graphicx}
\usepackage{caption}
\usepackage{subcaption}
\usepackage{comment}
\usepackage{hyperref}
\usepackage[capitalise, noabbrev]{cleveref}

\newcommand{\ve}[1]{\mathbf{#1}}
\newcommand{\norm}[1]{\left\| #1 \right\|}

\newtheorem{theorem}{Theorem}[section]
\newtheorem{lemma}[theorem]{Lemma}
\newtheorem{corollary}[theorem]{Corollary}

\newtheorem{definition}[theorem]{Definition}
\newtheorem{remark}[theorem]{Remark}
\newtheorem{fact}[theorem]{Fact}
\newtheorem{proposition}[theorem]{Proposition}

\newcommand{\jh}[1]{\textcolor{magenta}{#1}}

\title{Paving the Way for Consensus: Convergence of Block Gossip Algorithms}
\author{Jamie Haddock}
\address{Department of Mathematics, Harvey Mudd College, Claremont, CA 91711}
\thanks{JH and BJ are grateful to and were partially supported by NSF CAREER DMS \#1348721, NSF DMS \#2011140 and NSF BIGDATA DMS \#1740325 which are led by Professor Deanna Needell. JH is also partially supported by NSF DMS \#2111440.}
\email{jhaddock@g.hmc.edu}
\urladdr{https://jamiehadd.github.io/} 

\author{Benjamin Jarman}
\address{Department of Mathematics, University of California, Los Angeles, 
Los Angeles, CA 90095}
\email{bjarman@math.ucla.edu}
\thanks{}
\urladdr{https://www.math.ucla.edu/$\sim$bjarman/} 

\author{Chen Yap}
\address{Department of Mathematics, University of California, Los Angeles, 
Los Angeles, CA 90095}
\email{rkwyap@gmail.com}

\date{\today}

\begin{document}

\maketitle

\begin{abstract}
Gossip protocols are popular methods for average consensus problems in distributed computing. We prove new convergence guarantees for a variety of such protocols, including path, clique, and synchronous pairwise gossip. These arise by exploiting the connection between these protocols and the block randomized Kaczmarz method for solving linear systems. Moreover, we extend existing convergence results for block randomized Kaczmarz to allow for a more general choice of blocks, rank-deficient systems, and provide a tighter convergence rate guarantee. We furthermore apply this analysis to inconsistent consensus models and obtain similar guarantees. An extensive empirical analysis of these methods is provided for a variety of synthetic networks.
\end{abstract}

\section{Introduction}

Consider a network in which every node has a secret (unknown by the other nodes) value, and the goal is for all nodes to learn the average of these values.  This problem is a classical and fundamental problem in distributed computing and multi-agent systems known as \emph{average consensus} \cite{kempe2003gossip}; it has additional real-world applications in clock synchronization~\cite{freris2012fast}, PageRank, localization without GPS~\cite{zhang2019distributed}, opinion formation, distributed data fusion in sensor networks~\cite{xiao2005scheme}, blockchain technology, and load balancing~\cite{cybenko1989dynamic}. See Figure~\ref{fig:AC1} for a visualization of an average consensus problem.

Initial approaches for this problem may be to allow all nodes to pass their secret value to a single \emph{hub} node which would then perform the averaging and pass this value back to the others, or for every node to share its stored knowledge of all other nodes' secret values with its neighbours until all nodes have learned all stored values (a process known as \emph{flooding}~\cite{xiao2007distributed}).  However, such methods are problematic.  The first requires communication that may be infeasible as it may not respect the topology of the underlying network; in particular, there may be no hub node (see Figure~\ref{fig:AC1}).  The second may require many instances of communication between nodes and thus struggle to scale to modern large-scale networks. An attractive class of methods that go some way towards rectifying these issues are \emph{gossip} protocols, where at each time-step some subset of nodes are `activated' to share their stored information with each other across network edges. As is common in large-scale settings, such protocols are \emph{randomized} in the sense that activated nodes are often chosen at random, with distribution depending on the particular gossip protocol. Over many iterations, the values held at each node converge to the average over the whole network (with mild assumptions including connectivity of the underlying network) and the problem is solved in a distributed manner. Such methods have been a topic of popular study this side of the millenium, with the seminal 2006 paper of Boyd et al.~\cite{Boyd2006RandomizedGA} sparking a flurry of research on the topic~\cite{aysal2009broadcast,dimakis2010survey,dimakis2008geographic,hanzely2017privacy,liu2013analysis,nedic2018network,olshevsky2009convergence,zouzias2015randomized}.

\begin{figure}
         \centering
         \includegraphics[width=0.3\textwidth]{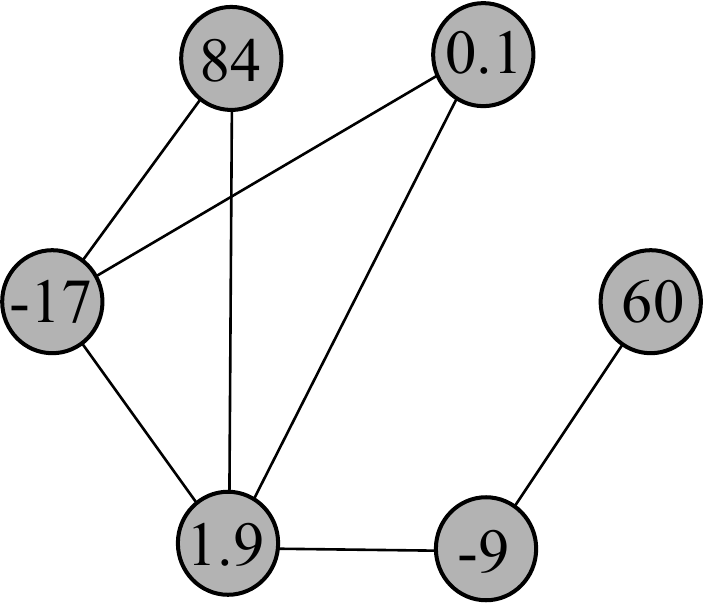}
         \caption{Average consensus problem with initial (secret) values listed. The consensus value for this problem is $\bar{c} = 20$.}
         \label{fig:AC1}
\end{figure}

In \cite{loizou2019revisiting}, Loizou and Richt\'arik united the study of randomized gossip protocols with randomized numerical linear algebra. They showed that a wide class of gossip protocols can be interpreted, under certain assumptions, as randomized iterative linear system solvers applied to a linear system derived from the network at hand. This remarkable connection prompted a variety of convergence results, new accelerated and weighted gossip variants, and dual edge-based gossip protocols, opening up the possibility for further links to be developed. In this paper, we hone in on the class of \emph{block} gossip algorithms discussed in their work: we show that several popular gossip protocols are members of this class, give strong convergence results for said methods, and generalize the theory on both the linear algebra and gossip protocol sides of the problem.

\subsection{Contributions}

Our main contributions in this work are threefold:
\begin{itemize}
    \item We generalize previous results on randomized block iterative methods for solving linear systems (see~\cite{necoara2019faster,needell2013paved,needell2015randomized}) to include the case where the system is less than full rank (which is vital for the average consensus case, as we will see), to include a wider class of block sampling protocols, and to sharpen the resulting convergence rate guarantee.
    \item We derive new convergence rates for popular gossip protocols such as \emph{path} \cite{pathgossip}, \emph{clique} \cite{cliquegossip}, and \emph{edge-independent set} \cite{Boyd2006RandomizedGA} gossiping, by showing they can be interpreted as block gossip methods. We furthermore generalize to include the case where multiple node sets can be activated simultaneously, and analyze the dependence of a protocol's performance on the spanning trees of the activated subgraphs.
    \item We give new analyses of gossip for \emph{inconsistent} consensus models. We link edge communication errors in a network to inconsistent, or noisy, linear systems, and connect the performance of iterative solvers on said systems to the performance of gossip protocols on said networks.
\end{itemize}

Furthermore, we provide a wide range of experiments to demonstrate the comparative performance of the discussed algorithms, both previously existing and new, on a variety of network structures.

We remark that we do not go into detail regarding the specific communication protocols that networks may have. For example, certain methods (e.g., path gossiping~\cite{pathgossip}) require multiple instances of information sharing at each iteration, which may not be implementable in all networks, and others (e.g., edge independent set gossiping~\cite{Boyd2006RandomizedGA}) require communication across each edge only once at each time step. 
We point the reader to \cite{Boyd2006RandomizedGA,sabergossip2007} for more details on network communication protocols and their use in gossip algorithms.


\subsection{Organization}

The rest of the paper is organized as follows. In Section~\ref{subsec:notation}, we introduce notation that will be used throughout the work. In Section~\ref{subsec:mainresults}, we describe in detail the methods considered in this paper and present our main theoretical results.  In Section~\ref{subsec:relatedwork}, we provide detailed background on the average consensus problem, and recent work on block iterative methods for solving linear systems. In Section~\ref{sec:proof}, we prove our generalization of the block iterative method theory, and go into detail on the connection between this and block gossip algorithms for average consensus. In Section~\ref{sec:BGsampling}, we connect particularly popular block gossip algorithms with our theory and produce explicit convergence rates for them. In Section~\ref{sec:inconsistent}, we explore the case of average consensus models with edge miscommunication and provide convergence results for the gossip method on this type of faulty model. Lastly, in Section~\ref{sec:experiments}, we provide numerous experiments and compare our theoretical results with the empirical behavior of the considered gossip methods.

\subsection{Notation and Definitions}\label{subsec:notation}
Let $\mathcal{G} = (\mathcal{V},\mathcal{E})$ denote an undirected network where $\mathcal{V}$ denotes the $n$ nodes of the network and $\mathcal{E}$ denotes the $m$ edges.  Let $\ve{Q} \in \mathbb{R}^{m \times n}$ denote the incidence matrix of the network.  Each row of $\ve{Q}$ corresponds to an edge of the network and each column corresponds to a node.  If row $l$ of $\ve{Q}$, denoted $\ve{q}_{l}^T$, corresponds to the edge $e_{ij} \in \mathcal{E}$ connecting node $i$ and $j$, then all but the $i$th and $j$th entries of $\ve{q}_{l}$ are zero with these entries containing a one and negative one (the order of the positive and negative entries does not matter).

We recall the notion of several special subgraph structures.  We remind the reader of the definition of an \emph{independent edge set}, that is a subset of edges of the graph in which no two edges are incident to the same node.  Additionally, we remind the reader that a \emph{clique subgraph}, or \emph{complete subgraph}, is a subset of edges of the graph which together form a subgraph in which every pair of nodes is connected; that is, the edge-induced subgraph is a clique.  Finally, we remind the reader of the definition of a \emph{path subgraph}, a subset of edges which together form a path graph; that is, the edge-induced subgraph is a path.

Throughout, we let $[m]$ denote the set of integers from one to $m$; $[m] := \{1, 2, \cdots, m\}$.  We additionally let $\ve{0}$ denote the vector of all zeros and let $\ve{1}$ denote the vector of all ones (the dimensions of these vectors will be given or obvious from context).

In what follows, we let $A_\tau$ denote the subset of matrix $A$ with rows indexed by the set $\tau$.  We let $\lambda_{\min}(A), \lambda_{\min +}(A)$, and $\lambda_{\max}(A)$ denote the minimum, minimum non-zero, and maximum eigenvalues of the matrix $A$, respectively.  Additionally, we denote by $A^\dagger$ the Moore-Penrose pseudoinverse of the matrix $A$.

We additionally recall the notion of a \emph{row paving}  of a matrix $\ve{A}$, which controls the conditioning of a set of submatrices that partition the matrix $\ve{A}$.  We include here the definition provided by Needell and Tropp in~\cite{needell2013paved}, and note that earlier work on the construction of pavings for block projection methods is due to Popa~\cite{Pop99:Block-Projections-Algorithms}. In the wider operator theory context, pavings have long been a topic of interest; see~\cite{weberpaving} for a review.

\begin{definition}
A $(d, \alpha, \beta)$ \emph{row paving} of a matrix $\ve{A}$ is a partition $T = \{\tau_1,\tau_2, \cdots, \tau_d\}$ of the row indices that satisfies $$\alpha \le \lambda_{\min}(\ve{A}_\tau \ve{A}_\tau^\top) \text{ and } \lambda_{\max}(\ve{A}_\tau \ve{A}_\tau^\top) \le \beta \text{ for each } \tau \in T.$$
\end{definition}

We will additionally be interested in subsets of the row indices that do not necessarily partition the rows.  We define a \emph{row covering} to be the natural relaxation of a row paving in which the requirement that the subsets partition the row indices be relaxed and where the parameter $\alpha$ provides a lower-bound for the minimum \emph{non-zero} eigenvalue rather than the minimum eigenvalue.  Each of these are important generalizations for our application of interest, gossip algorithms for average consensus.

\begin{definition}
A $(d, \alpha, \beta, r, R)$ \emph{row covering} of a matrix $\ve{A}$ is a collection of subsets $T = \{\tau_1,\tau_2, \cdots, \tau_d\}$ of the row indices, $\tau_i \subset [m]$ for all $i = 1, \cdots, d$, that covers the row indices, for each $i \in [m]$ we have $i \in \tau_l$ for some $l = 1, \cdots, d$, and that satisfies $$\alpha \le \lambda_{\min +}(\ve{A}_\tau \ve{A}_\tau^\top) \text{ and } \lambda_{\max}(\ve{A}_\tau \ve{A}_\tau^\top) \le \beta \text{ for each } \tau \in T,$$ where $r$ and $R$ are the minimum and maximum, respectively, number of blocks in which a single row appears, i.e., $r = \min_{i \in [m]} |\{\tau_l \in T: i \in \tau_l\}|$ and $R = \max_{i \in [m]} |\{\tau_l \in T: i \in \tau_l\}|$.
\end{definition}


First, note that the minimum and maximum number of repeated occurrences of a row in the blocks, $r$ and $R$, satisfy $$1 \le r \le R \le d.$$ Further, note that a $(d, \alpha, \beta)$ row paving is by definition a $(d, \alpha, \beta, 1, 1)$ row covering.  Next, we remind the readers that \emph{every} sufficiently tall row-normalized matrix $\ve{A}$ admits a row paving with $\alpha > 0$ and $\beta < 2$~\cite{needell2013paved}; the authors illustrate that additional mild assumptions on the matrix allows one to produce such a row paving via random partitioning.  Additionally, we note that if the matrix in question is an incidence matrix $\ve{Q}$, where rows correspond to edges in a graph $\mathcal{G} = (\mathcal{V},\mathcal{E})$, then the row covering $T = \{\tau_1, \cdots, \tau_d\}$ corresponds to a collection of sets of edges such that every edge appears in at least one set; that is $\tau_i \subset \mathcal{E}$ for all $i \in [d]$ and for all $e \in \mathcal{E}$, we have $e \in \tau_l$ for some $l \in [d]$.

\subsection{Main Results}\label{subsec:mainresults}

The average consensus problem is defined over an undirected network $\mathcal{G} = (\mathcal{V},\mathcal{E})$.
We let $\ve{c} = (c^{(1)}, c^{(2)}, ..., c^{(n)})^T$ denote the vector of secret values (i.e., $c^{(i)}$ is the secret value of the $i$th node) initally held by the nodes of the network.
The average consensus problem is then to ensure, after some communication protocol is applied, that each node stores the averaged value $\bar{c} := \text{mean}(\ve{c})$; that is the final vector of updated node values is $\ve{c}^* = \bar{c}\ve{1}$ where $\ve{1} \in \mathbb{R}^n$ is the vector of all ones.  We can formulate this problem as:
\begin{equation}
\text{Find } \ve{c}^* = \text{argmin}_{\ve{x} \in \mathbb{R}^n} \|\ve{x} - \ve{c}\|^2 \text{ s.t. } \ve{Q}\ve{x} = \ve{0}.  \label{AC}
\end{equation} 
Here $\ve{0} \in \mathbb{R}^{|\mathcal{E}|}$ denotes the all zeros vector. The average consensus problem may be formulated in this way using either the incidence matrix or the Laplacian matrix, $L = D - A$ where $D$ is the diagonal matrix of node degrees and $A$ is the adjacency matrix, or more generally as an \emph{average consensus system} as defined in \cite{loizou2016new}. In this work, we will focus upon the \emph{block gossip} methods for this problem, in which groups of edges are sampled and nodes connected by these edges update their value to their average (or the average of a subset of these nodes); pseudocode for this method is provided in Algorithm~\ref{alg:BG}.  See Figure~\ref{fig:BGsteps} for a visualization of one step of Algorithm~\ref{alg:BG} on the average consensus problem of Figure~\ref{fig:AC1} with independent edge set, path, and clique block sampling.

We next present our main result which illustrates that the block gossip method converges at least linearly in expectation to consensus if the underlying graph is connected.  Moreover, we specialize this result to three special cases: the case when the blocks of edges sampled in each iteration are independent edge sets, when they are clique subgraphs, and when they are path subgraphs, and give refined convergence rates for these cases.  

\begin{corollary}\label{cor:BG}
Suppose graph $\mathcal{G} = (\mathcal{V},\mathcal{E})$ is connected, $\ve{Q} \in \mathbb{R}^{|\mathcal{E}| \times |\mathcal{V}|}$ is the incidence matrix for $\mathcal{G}$, and $T = \{\tau_1, \cdots,\tau_d\}$ is a $(d,\alpha,\beta,r,R)$ row covering for $\ve{Q}$ with $M = \max_{i \in [d]}|\tau_i|$.  Then the block gossip method with blocks determined by $T$ converges at least linearly in expectation with the guarantee $$\mathbb{E}\norm{\ve{c}_k - \ve{c}^*}^2 \le \left(1 - \frac{r\alpha(\mathcal{G})}{\beta d}\right)^k \norm{\ve{c} - \ve{c}^*}^2,$$ where $\alpha(\mathcal{G})$ is the algebraic connectivity of graph $\mathcal{G}$. 
\begin{enumerate}
    \item If $T$ consists of independent edge sets, the rate constant can be bounded by $$\left(1 - \frac{r\alpha(\mathcal{G})}{2d}\right).$$
    \item If $T$ consists of path subgraphs, the rate constant can be bounded by $$\left(1 - \frac{r\alpha(G)}{(2 - 2\cos\frac{M\pi}{M+1})d}\right) \le \left(1 - \frac{r\alpha(\mathcal{G})}{4d}\right).$$
    \item If $T$ consists of clique subgraphs, the rate constant can be bounded by $$\left(1 - \frac{r\alpha(G)}{(2 - 2\cos\frac{M\pi}{M+1})d}\right) \le \left(1 - \frac{r\alpha(\mathcal{G})}{4d}\right).$$
    \item If $T$ consists of arbitrary connected subgraphs, the rate constant can be bounded by $$\left(1 - \frac{r\alpha(\mathcal{G})}{Md}\right).$$
\end{enumerate}
\end{corollary}

\begin{algorithm}[H]
	\caption{Block Gossip Method}
	\begin{algorithmic}[1]
		\Procedure{BG}{$\mathcal{G},\ve{c}_0 = \ve{c},T = \{\tau_1,\cdots,\tau_d\}$}
		\State $k = 0$
		\Repeat
		\State $k \leftarrow k+1$
		\State Choose edge subset $\tau$ uniformly at random from elements of $T$. \label{line:selection}
		\State Form $\mathcal{G}_{\tau} = (\mathcal{V}_{\tau},\mathcal{E}_{\tau})$, the edge-induced subgraph of $\mathcal{G}$ defined by edges in $\tau$. 
		\State {Nodes outside of $\mathcal{V}_\tau$ do not update secret values, $(\ve{c}_k)_{\mathcal{V}_{\tau}^C} \leftarrow (\ve{c}_{k-1})_{\mathcal{V}_{\tau}^C}$.}
		\ForAll{connected components $\mathcal{G}' = (\mathcal{V}',\mathcal{E}')$ of $\mathcal{G}_\tau$}
		\State Nodes in $\mathcal{V}'$ solve average consensus on $\mathcal{G}'$, $(\ve{c}_k)_{\mathcal{V}'} \leftarrow \left[\frac{1}{|\mathcal{V}'|}\sum_{i\in\mathcal{V}'} (\ve{c}_k)_i \right] \ve{1}_{|\mathcal{V}'|}$. 
		\EndFor
		\Until{stopping criterion reached}
		\State \textbf{return} $\ve{c}_k$ 
		\EndProcedure
	\end{algorithmic}
	\label{alg:BG}
\end{algorithm}

\begin{figure}[H]
     \begin{subfigure}[b]{0.3\textwidth}
         \centering
         \includegraphics[width=0.9\textwidth]{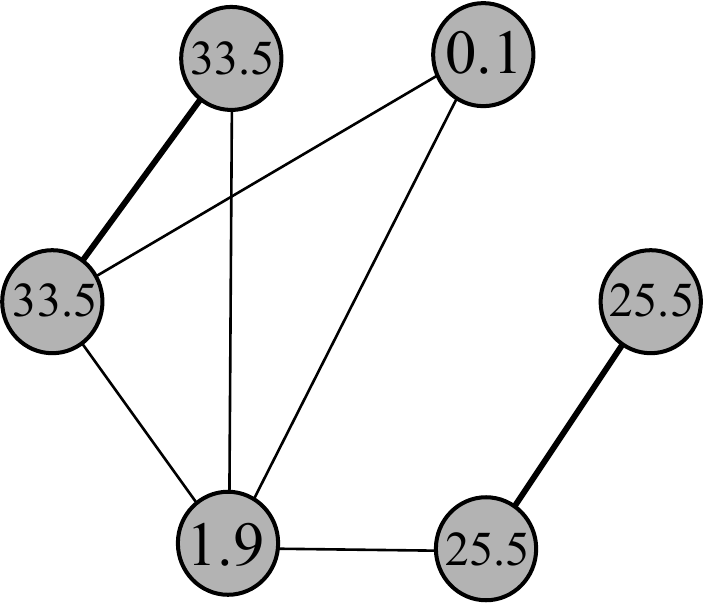}
         \caption{One iteration of block gossip with a block corresponding to an independent edge set of size 2.}
         \label{fig:AC2}
     \end{subfigure}\hfill
     \begin{subfigure}[b]{0.3\textwidth}
         \centering
         \includegraphics[width=0.9\textwidth]{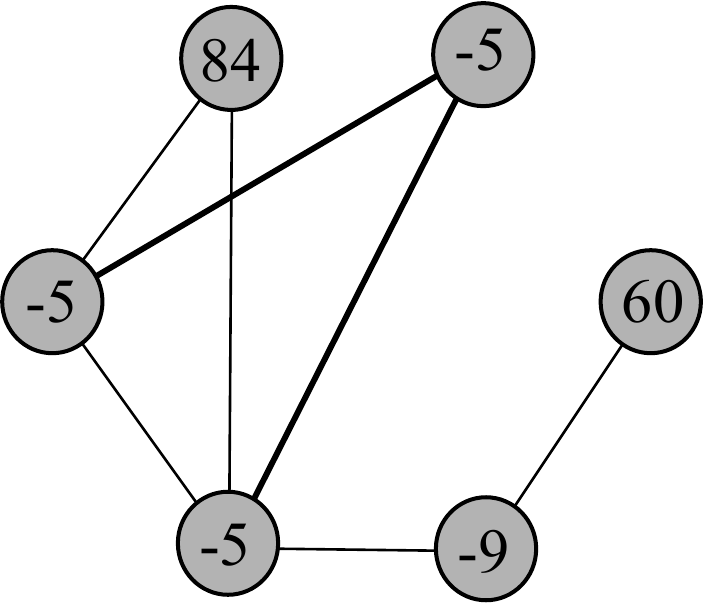}
         \caption{One iteration of block gossip with a block corresponding to a path of length 2.}
         \label{fig:AC3}
     \end{subfigure}\hfill
     \begin{subfigure}[b]{0.3\textwidth}
         \centering
         \includegraphics[width=0.9\textwidth]{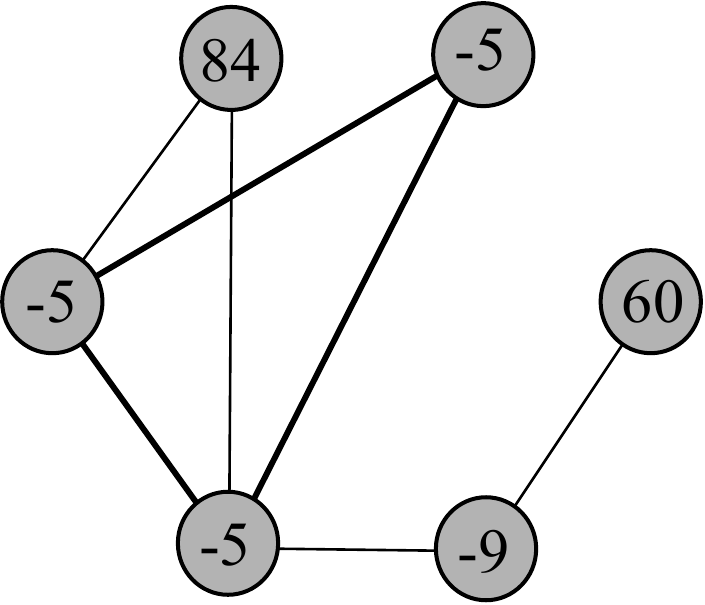}
         \caption{One iteration of block gossip with a block corresponding to a clique of size 3.}
         \label{fig:AC4}
     \end{subfigure}
    \caption{The average consensus problem of Figure~\ref{fig:AC1} after an iteration of the block gossip method (Algorithm~\ref{alg:BG}) with various types of block structures. The edges defining the sampled block are represented by bold lines.}
    \label{fig:BGsteps}
\end{figure}

This result follows from a new bound on the convergence rate of the \emph{block Kaczmarz method} on a potentially rank-deficient least-squares problem.  The block Kaczmarz method samples blocks of rows of the matrix $\ve{A}$ in each iteration and performs an update which projects the previous iterate onto the solution space of the subset of sampled equations; note that the standard single-row Kaczmarz updates are a special case of block Kaczmarz with block size one.  The details of this method are provided in Algorithm~\ref{alg:BK}.  The block gossip method with blocks $T$ produces the same iterates as the block Kaczmarz method performed with $\ve{A} = \ve{Q}, \ve{b} = \ve{0}$, and $\ve{x}_0 = \ve{c}$ with blocks $T$.  


\begin{algorithm}
	\caption{Block Kaczmarz Method}
	\begin{algorithmic}[1]
		\Procedure{BK}{$\ve{A},\ve{b},\ve{x}_0,T = \{\tau_1,\cdots,\tau_d\}$}
		\State $k = 0$
		\Repeat
		\State $k \leftarrow k+1$
		\State Choose row block $\tau$ uniformly at random from $T$. \label{line:selection}
		\State $\ve{x}_k \leftarrow \ve{x}_{k-1} + \ve{A}_{\tau}^\dagger(\ve{b}_\tau - \ve{A}_\tau\ve{x}_{k-1})$
		\Until{stopping criterion reached}
		\State \textbf{return} $\ve{x}_k$ 
		\EndProcedure
	\end{algorithmic}
	\label{alg:BK}
\end{algorithm}

Our main result regarding the block Kaczmarz method generalizes the main result of~\cite{needell2013paved} in several ways:
\begin{itemize}
    \item Generalizes to the case when the least-squares problem is rank-deficient.
    \item Relaxes the requirement that the row blocks be sampled from a matrix paving.
    \item Demonstrates that the convergence horizon depends upon the minimum \emph{nonzero} singular value of the blocks $\ve{A}_\tau$ rather the absolute minimum singular value (often 0).
\end{itemize}
These generalizations are important in our main application to the average consensus problem and block gossip methods, but are likely of independent interest in other applications.

\begin{theorem}\label{thm:main}
Consider the least-squares problem $$\min \|\ve{A}\ve{x} - \ve{b}\|^2$$ where $\ve{A} \in \mathbb{R}^{m \times n}$ is not necessarily full-rank and $\ve{b} \in \mathbb{R}^m$.  
Let $\ve{e} = A\ve{x}_{\ve{e}} - \ve{b}$ for some $\ve{x}_{\ve{e}}$ and
let $\{\tau_1, \cdots, \tau_d\}$ be a $(d, \alpha, \beta, r, R)$ covering (not necessarily a paving) of the rows of $\ve{A}$. Let $\ve{x}_j$ denote the $j$th iterate produced by Block RK on the system defined by $\ve{A}$ and $\ve{b}$ with initial iterate $\ve{x}_0$, and let $\ve{x}^\ast := \left(\ve{I} - \ve{A}^\dagger \ve{A}\right)\ve{x}_0 + \ve{A}^{\dagger}(\ve{b} + \ve{e}) = \left(\ve{I} - \ve{A}^\dagger \ve{A}\right)\ve{x}_0 + \ve{A}^{\dagger}\ve{A}\ve{x}_{\ve{e}}$. 
Then we have
\begin{equation}
    \mathbb{E}\left(\norm{\ve{x}_{j} - \ve{x}^\ast}^2\right) \leq \left(1 - \frac{r\sigma_{\min +}^2(\ve{A})}{\beta d}\right)^{j} \norm{\ve{x}_0 - \ve{x}^\ast}^2 + \frac{\beta R}{\alpha r \sigma_{\min +}^2(\ve{A})} \norm{\ve{e}}^2,
\end{equation}
where $\sigma_{\min +}(\ve{A})$ is the smallest nonzero singular value of $\ve{A}$.  
\end{theorem}

\begin{remark}
The convergence horizon term, $\frac{\beta R}{\alpha r \sigma_{\min +}^2(\ve{A})} \norm{\ve{e}}^2$, is minimized in the case that $\ve{e} = \ve{e}^*$ where $\ve{e}^*$ is the minimum norm residual, i.e., $$\ve{e}^* = \text{argmin } \|\ve{e}\|^2 \text{ s.t. } \ve{e} = \ve{A}\ve{x} - \ve{b} \text{ for some } \ve{x} \in \mathbb{R}^n.$$ In this case, the iterates converge to $\ve{x}^\ast := \left(\ve{I} - \ve{A}^\dagger \ve{A}\right)\ve{x}_0 + \ve{A}^{\dagger}(\ve{b} + \ve{e}^*)$.
\end{remark}

\begin{remark}
In other Kaczmarz literature, including Needell's original result for Randomized Kaczmarz applied to inconsistent linear systems~\cite[Theorem 2.1]{Nee10:Randomized-Kaczmarz}, the problem is formulated by adding some vector of measurement noise $\ve{r}$ to a consistent system $\ve{A}\ve{x} = \ve{b}$, leading to a convergence horizon term proportional to $\norm{\ve{r}}$. We instead work with $\ve{A}\ve{x} = \ve{b}$ being inconsistent and give a horizon proportional to a residual, but we note that the two formulations are equivalent.
\end{remark}

\begin{remark}
We include a generalization of Theorem~\ref{thm:main} to the case when the right hand side, $\ve{b}$, is varying in each iteration according to mean-zero randomly distributed additive noise in Proposition~\ref{prop:randomNoiseBRK}.
\end{remark}

\subsection{Related Work} \label{subsec:relatedwork}

In this subsection we offer some related reading in the fields of average consensus, gossip protocols, and block iterative methods, and draw connections between them and our work.

\indent \textbf{Average consensus and gossip protocols.} As mentioned, average consensus has been a fundamental topic in distributed computing since the inception of the field. We refer the reader to the classical work of DeGroot \cite{degroot1974} for an inception of the consensus problem, and to the work of Tsitsiklis, Bertsekas, and Athans \cite{tsitsiklis1986} for a first look into stochastic protocols for distributed computing. As networks have grown larger in size and have appeared in more applications, the need for more efficient average consensus solvers motivated the development of \emph{gossip algorithms}: protocols that, in general, select some subset of nodes and allow them to `gossip', i.e., share and average their stored values. Boyd et al.'s 2006 paper~\cite{Boyd2006RandomizedGA} provided a fundamental exposition of said protocols, particularly on the connection between their convergence rate and the underlying network topology, and has motivated research in the topic ever since~\cite{dimakis2010survey}.

There have been a variety of works analyzing different node-selection protocols for gossiping. In \cite{pathgossip}, at each epoch a path of nodes in the network is formed and the values along said path are averaged. In~\cite{cliquegossip}, the network is decomposed beforehand into cliques (connected subgraphs), and at each epoch one such clique is activated. Lastly in~\cite{Boyd2006RandomizedGA}, at each epoch a selection of pairs are chosen and each pair computes its own average (we call this \textit{edge independent set} gossiping throughout). These analyses are, however, somewhat disjoint, so we believe the unified convergence analysis presented in this paper (which covers all of the aforementioned methods) to be novel.

\textbf{Block randomized Kaczmarz.} The Kaczmarz method \cite{ogkacz} is an iterative linear system solver whose popularity boomed after its randomized variant was proven to have exponential convergence by Strohmer and Vershynin~\cite{vershstrohkacz}. After this work proved convergence for full-rank, consistent systems, further work was done to generalize to the case of inconsistent~\cite{Nee10:Randomized-Kaczmarz} and rank-deficient~\cite{rek} systems. 

A well-studied family of variants are \textit{block Kaczmarz methods}, in which iterates are projected onto subspaces corresponding to blocks of rows rather than single equations. Early references include \cite{Eggermont1981IterativeAF, Elf80:Block-Iterative-Methods, Pop99:Block-Projections-Algorithms}, but we focus on the block randomized Kaczmarz method introduced by Needell and Tropp \cite{needell2013paved}. The authors proved that under certain restrictions on the choice of blocks, the method achieves exponential convergence, and converges up to a threshold if the system is inconsistent. In our work we significantly relax these conditions and achieve a similar convergence guarantee.

The connection between gossip algorithms and Kaczmarz methods for linear systems was analyzed in depth by Loizou and Richt\'arik in their 2019 paper \cite{loizou2019revisiting}; others considering this connection include~\cite{HM19Greed,zouzias2015randomized}. This connection was exploited in~\cite{loizou2019revisiting} to build a framework giving new convergence guarantees for gossip protocols, accelerations via momentum, and other interesting discussions such as dual gossip algorithms. In our work we hone in on their general exposition of block gossip algorithms, and give more explicit links and convergence guarantees for previously mentioned existing gossip protocols.

\section{Convergence of Block Randomized Kaczmarz}\label{sec:proof}


In this section, we prove our main result, Theorem~\ref{thm:main}, which illustrates that the block randomized Kaczmarz method converges at least linearly in expectation on least-squares problems even in the case that the matrix $\ve{A}$ is rank-deficint and the blocks are not sampled from a matrix paving.  We will then illustrate how this result specializes to prove Corollary~\ref{cor:BG}.

\begin{proof}[Proof of Theorem~\ref{thm:main}]
    First, note that by definition, $\ve{e} = \ve{A}\ve{x}_{\ve{e}} - \ve{b}$ and so
    \begin{align}\label{eq:errorcalc}
        \ve{A}\ve{x}^* - \ve{b} &= (\ve{A}\ve{x}_0 - \ve{A}\ve{A}^\dagger\ve{A}\ve{x}_0) + \ve{A}\ve{A}^\dagger\ve{b} + \ve{A}\ve{A}^\dagger\ve{e} - \ve{b}\nonumber
        \\&= \ve{A}\ve{A}^\dagger\ve{b} + \ve{A}\ve{A}^\dagger(\ve{A}\ve{x}_\ve{e} - \ve{b}) - \ve{b}
        \\&= \ve{A}\ve{x}_\ve{e} - \ve{b} = \ve{e},\nonumber
    \end{align}
    where the second and third equation used the fact that $\ve{A}\ve{A}^\dagger \ve{A} = \ve{A}$. 
    
   Recall that our updates take the form $\ve{x}_{j+1} = \ve{x}_j + \ve{A}^\dagger_{\tau}(\ve{b}_\tau - \ve{A}_\tau \ve{x}_{j})$, where $\tau$ is chosen uniformly at random from our set of blocks. We then have
   \begin{align}\label{eq:errorsquared}
       \norm{\ve{x}_{j+1} - \ve{x}^\ast}^2 &= \norm{\ve{x}_j + \ve{A}^\dagger_\tau (\ve{b}_\tau - \ve{A}_\tau \ve{x}_{j}) - \ve{x}^\ast}^2 \nonumber\\
       &= \norm{(\ve{I} - \ve{A}_\tau^\dagger \ve{A}_\tau)\ve{x}_j + \ve{A}_{\tau}^\dagger (\ve{A}_\tau \ve{x}^* - \ve{e}_\tau) - \ve{x}^*}^2\\
       &= \norm{(\ve{I} - \ve{A}^\dagger_\tau \ve{A}_\tau)(\ve{x}_j - \ve{x}^\ast) - \ve{A}_\tau^\dagger \ve{e}_\tau}^2 \nonumber\\
       &= \norm{(\ve{I} - \ve{A}^\dagger_\tau \ve{A}_\tau)(\ve{x}_j - \ve{x}^\ast)}^2 + \norm{\ve{A}_\tau^\dagger \ve{e}_\tau}^2,\nonumber
   \end{align}
   where we used that $\ve{A}^\dagger_\tau (\ve{b}_\tau + \ve{e}_\tau) = \ve{A}^\dagger_\tau \ve{A}_\tau \ve{x}^\ast$ and that $\operatorname{Im}(\ve{A}_\tau^\dagger)$ and $\operatorname{Im}(\ve{I} - \ve{A}_\tau^\dagger \ve{A}_\tau)$ are orthogonal. Then since $\ve{I} - \ve{A}^\dagger_\tau \ve{A}_\tau$ is an orthogonal projector, we have
   \begin{equation}\label{eq:breakuperror}
       \norm{(\ve{I} - \ve{A}^\dagger_\tau \ve{A}_\tau)(\ve{x}_j - \ve{x}^\ast)}^2 = \norm{\ve{x}_j - \ve{x}^\ast}^2 - \norm{\ve{A}^\dagger_\tau \ve{A}_\tau(\ve{x}_j - \ve{x}^\ast)}^2.
   \end{equation}
   Note that $\norm{\ve{A}_\tau^\dagger \ve{e}_\tau}^2 \le \sigma_{\max}^2(\ve{A}_\tau^\dagger) \norm{\ve{e}_\tau}^2 = \frac{1}{\sigma_{\min +}^2(\ve{A}_\tau)} \norm{\ve{e}_\tau}^2 \le \frac{1}{\alpha}\norm{\ve{e}_\tau}^2$. 
   Using this fact and taking expectations, we obtain
   \begin{align}
   \label{errorestim}
       \mathbb{E}_j\left(\norm{\ve{x}_{j+1} - \ve{x}^\ast}^2\right) &\le \norm{\ve{x}_j - \ve{x}^\ast}^2 - \mathbb{E}_j\left(\norm{\ve{A}^\dagger_\tau \ve{A}(\ve{x}_j - \ve{x}^\ast)}^2\right) + \frac{1}{\alpha} \mathbb{E}_j \left(\norm{\ve{e}_\tau}^2\right)\nonumber\\
       &\leq \norm{\ve{x}_j - \ve{x}^\ast}^2 - \frac{1}{\beta}\mathbb{E}_j \left(\norm{\ve{A}_\tau(\ve{x}_j - \ve{x}^\ast)}^2\right) + \frac{1}{\alpha} \mathbb{E}_j \left(\norm{\ve{e}_\tau}^2\right)\\
       &\le \norm{\ve{x}_j - \ve{x}^\ast}^2 - \frac{r}{\beta d}\norm{\ve{A}(\ve{x}_j - \ve{x}^\ast)}^2 + \frac{R}{\alpha d} \norm{\ve{e}}^2,\nonumber
   \end{align}
   where the last inequality follows from the fact that $$\mathbb{E}_j(\norm{\ve{v}_\tau}^2) = \frac{1}{d} \sum_{l=1}^d \norm{\ve{v}_{\tau_l}}^2 = \frac{1}{d} \sum_{l=1}^d \sum_{i=1}^m \ve{1}(i \in \tau_l) v_i^2 = \frac{1}{d} \sum_{i=1}^m \left[\sum_{l=1}^d \ve{1}(i\in \tau_l)\right] v_i^2$$ so we have $\frac{r}{d}\norm{\ve{v}}^2 \le \mathbb{E}_j(\norm{\ve{v}_\tau}^2) \le \frac{R}{d}\norm{\ve{v}}^2$.
   
   We now claim that for all $j$, $\ve{x}_j - \ve{x}^\ast \in \operatorname{Im}(\ve{A}^T)$, i.e., the row space of $\ve{A}$. We do so by induction; firstly for $j=0$ we have
   \begin{equation*}
       \ve{x}_0 - \ve{x}^\ast = \ve{x}_0 - (\ve{I} - \ve{A}^\dagger \ve{A})\ve{x}_0 - \ve{A}^\dagger (\ve{b} + \ve{e}) = \ve{A}^\dagger(\ve{A}\ve{x}_0 - (\ve{b} + \ve{e})),
   \end{equation*}
   and since $\operatorname{Im}(\ve{A}^\dagger) = \operatorname{Im}(\ve{A}^T)$, we are done.
   
   Now assume $\ve{x}_l - \ve{x}^\ast \in \operatorname{Im}(\ve{A}^T)$. Then we have, for some $\tau$, 
   \begin{equation*}
       \ve{x}_{l+1} - \ve{x}^\ast = \ve{x}_l + \ve{A}_\tau^\dagger(\ve{b}_\tau - \ve{A}_\tau \ve{x}_l) - \ve{x}^\ast.
   \end{equation*}
   By assumption $\ve{x}_l - \ve{x}^\ast \in \operatorname{Im}(\ve{A}^T)$, and furthermore since $\ve{A}_\tau$ is a row submatrix of $\ve{A}$, we have
   \begin{equation*}
       \operatorname{Im}(\ve{A}^\dagger_\tau) = \operatorname{Im}(\ve{A}^T_\tau) \subseteq \operatorname{Im}(\ve{A}^T).
   \end{equation*}
   Thus $\ve{x}_{l+1} - \ve{x}^\ast \in \operatorname{Im}(A^T)$, and we are done by induction.
   
   Now, returning to \eqref{errorestim}, since $\ve{x}_j - \ve{x}^\ast \in \operatorname{Im}(\ve{A}^T) = \operatorname{Ker}(\ve{A})^\perp$, we have
   \begin{equation*}
       \norm{\ve{A}(\ve{x}_j - \ve{x}^\ast)}^2 \geq \sigma^2_{\min +}(\ve{A})\norm{\ve{x}_j - \ve{x}^\ast}.
   \end{equation*}
   This yields
   \begin{equation*}
       \mathbb{E}_j\left(\norm{\ve{x}_{j+1} - \ve{x}^\ast}^2\right) \leq \left(1 - \frac{r\sigma_{\min +}^2(\ve{A})}{\beta d}\right)\norm{\ve{x}_j - \ve{x}^\ast}^2 + \frac{R}{\alpha d} \norm{\ve{e}}^2,
   \end{equation*}
   and by induction we obtain
   \begin{align}\label{eq:induction}
       \mathbb{E}\left(\norm{\ve{x}_{j+1} - \ve{x}^\ast}^2\right) &\leq \left(1 - \frac{r\sigma_{\min +}^2(\ve{A})}{\beta d}\right)^{j+1}\norm{\ve{x}_0 - \ve{x}^\ast}^2 + \left[\sum_{i=0}^j \left(1 - \frac{r\sigma_{\min +}^2(\ve{A})}{\beta d}\right)^i\right] \frac{R}{\alpha d} \norm{\ve{e}}^2 \nonumber\\
       &\le \left(1 - \frac{r\sigma_{\min +}^2(\ve{A})}{\beta d}\right)^{j+1}\norm{\ve{x}_0 - \ve{x}^\ast}^2 + \left[\sum_{i=0}^\infty \left(1 - \frac{r\sigma_{\min +}^2(\ve{A})}{\beta d}\right)^i\right] \frac{R}{\alpha d} \norm{\ve{e}}^2 \\
       &= \left(1 - \frac{r\sigma_{\min +}^2(\ve{A})}{\beta d}\right)^{j+1}\norm{\ve{x}_0 - \ve{x}^\ast}^2 + \frac{\beta R}{\alpha r \sigma_{\min +}^2(\ve{A})} \norm{\ve{e}}^2.\nonumber
   \end{align}
\end{proof}

\begin{remark}
We note that in most applications, including our application of average consensus, $\sigma_{\min +}(\ve{A})$ is fixed, and so it is natural to seek to maximize $\frac{r}{\beta d}$. For the average consensus problem, this equates to careful selection of the block set $T$.
\end{remark}


We next include a proof of Corollary~\ref{cor:BG} which follows from Theorem~\ref{thm:main} due to the fact that block gossip on the network $\mathcal{G}$ with initial secret node values $\ve{c}$ coincides with block randomized Kaczmarz on the problem~\eqref{AC}.
\begin{proof}[Proof of Corollary~\ref{cor:BG}]
Consider running block RK with $\ve{A} = \ve{Q}, \ve{b} = \ve{0}$, and $\ve{x}_0 = \ve{c}$ with samples $\tau \in T$ determined by the run of block gossip on $\mathcal{G}$ with $\ve{c}_0 = \ve{c}$.  Note that $\operatorname{ker}(\ve{Q})$ is nonempty since $\operatorname{rank}(\ve{Q}) \le n-1$, so $\ve{e} = \ve{0}$.  It follows from \cite[Theorem 5]{loizou2019revisiting} that block gossip iterate $\ve{c}_k$ and block RK iterate $\ve{x}_k$ coincide for all $k$.


If we take $\ve{x}_0 = \ve{c}$, then by Theorem \ref{thm:main} we know that block RK will converge to $(\ve{I} - \ve{Q}^\dagger \ve{Q})\ve{c}$ (since $\ve{b} = \ve{e} = \ve{0}$ in this application). This is exactly the orthogonal projection of $\ve{c}$ onto $\operatorname{ker}(\ve{Q}) = \operatorname{span}\{\ve{1}\}$, where $\ve{1}$ is the length-$|\mathcal{E}|$ vector of all ones. Then, since $\left\{\frac{1}{\sqrt{|\mathcal{E}|}}\ve{1}\right\}$ is an orthonormal basis for $\operatorname{ker}(\ve{Q})$, we can compute this projection as 
\begin{align}
    (\ve{I} - \ve{Q}^\dagger \ve{Q})\ve{c} &= \left\langle \ve{c}, \frac{1}{\sqrt{|\mathcal{E}|}}\ve{1} \right\rangle \frac{1}{\sqrt{|\mathcal{E}|}}\ve{1} \\
    &= \frac{1}{|\mathcal{E}|}\left(\sum_{i=1}^{|\mathcal{E}|} c_i\right) \ve{1} \\
    &= \ve{c}^\ast.
\end{align}
%
Finally, note that $\sigma_{\min +}^2(\ve{Q}) = \lambda_{\min +}(\ve{L}) = \alpha(\mathcal{G})$ where $\ve{L}$ is the Laplacian matrix of $\mathcal{G}$.  The specific results enumerated in Corollary~\ref{cor:BG} follow from the singular value upper bounds presented in Section~\ref{sec:BGsampling}.
\end{proof}



\section{Block Gossip Sampling}\label{sec:BGsampling}

In this section, we consider particular cases when the blocks used in the block gossip method correspond to special subgraph structures, namely independent edge sets, clique subgraphs, path subgraphs, and arbitrary connected subgraphs.


We begin with a lemma that will be used to strengthen our convergence results for these cases.

\begin{lemma}
\label{spanninglemma}
Let $\tau$ be a subset of edges of $\mathcal{G}$, and let $\tau' \subseteq \tau$ be the edge set of a spanning tree of $\mathcal{G}_\tau$. Then the block gossip updates produced by choosing blocks $\tau$ and $\tau'$ are identical.
\end{lemma}

\begin{proof}
This follows immediately from the fact that $\mathcal{G}_\tau$ and $\mathcal{G}_{\tau'}$ have the same vertex set, and block gossip simply averages the stored values of said vertex set at each iteration.
\end{proof}

The value of this lemma comes from the fact that our convergence rate depends (in part) on the maximum singular values of our blocks: removing rows from a matrix (i.e., using a subset of edges in our block) decreases said singular values, improving the convergence rate.


Throughout this section, we make use of the fact that for any collection of row indices $\tau$, the spectra of $\ve{Q}_\tau \ve{Q}_\tau^\top$ and $\ve{L}_\tau$ are the same, up to zeros. In particular, we can identify our covering constants $\alpha$ and $\beta$ by analyzing the spectrum of $\ve{L}_\tau$, which is well understood for many of the graph structures we will consider.


\subsection{Independent edge set blocks}

In the case that $T = \{\tau_1, \tau_2, \cdots, \tau_d\}$ is a row covering of $\ve{Q}$ where each $\mathcal{G}_{\tau_i}$ is an independent edge set, we have that $$\ve{Q}_{\tau_i}\ve{Q}_{\tau_i}^\top = 2\ve{I}$$ for each $i \in [d].$ Thus, $T$ is a $(d, 2, 2, r, R)$ row covering of $\ve{Q}$ and we have that $$\left(1 - \frac{r\alpha(\mathcal{G})}{\beta d}\right) = \left(1 - \frac{r\alpha(\mathcal{G})}{2 d}\right).$$

Note that each independent each set is its own spanning tree, so there is no improvement to be made here via Lemma~\ref{spanninglemma}.

\subsection{Path blocks}

In the case that $T = \{\tau_1, \tau_2, \cdots, \tau_d\}$ is a row covering of $\ve{Q}$ where each $\mathcal{G}_{\tau_i}$ is a path, we make use of the following fact about the eigenvalues of the Laplacian of a path subgraph (see e.g., \cite{alggraphtheorybook}).

\begin{fact}
Let $P_n$ be a path subgraph of $G$ of length $n$. Then the eigenvalues of $L_{P_n}$ are
\begin{equation}
    2 - 2\cos\frac{\pi k}{n} \quad \text{  for $k = 0, 1, \cdots, n-1$.}
\end{equation}
\end{fact}

We then have for each $i \in [d]$ that
\begin{align}
    \lambda_{\max}(\ve{Q}_{\tau_i}\ve{Q}_{\tau_i}^\top) &= \lambda_{\max}(\ve{Q}_{\tau_i}^\top \ve{Q}_{\tau_i}) = 2 - 2\cos \frac{|\tau_i|\pi}{|\tau_i|+1} \le 2 - 2\cos \frac{M\pi}{M+1},
\end{align}
so $T$ is a $(d, \alpha, 2 - 2\cos \frac{M\pi}{M+1}, r, R)$ row covering.  Thus we have that $$\left(1 - \frac{r\alpha(\mathcal{G})}{\beta d}\right) = \left(1 - \frac{r\alpha(\mathcal{G})}{(2 - 2\cos \frac{M\pi}{M+1}) d}\right) \le \left(1 - \frac{r\alpha(\mathcal{G})}{4 d}\right).$$

Again, each path is its own spanning tree, so this rate cannot be improved via Lemma~\ref{spanninglemma}.

\subsection{Clique blocks}

In the case that $T = \{\tau_1, \tau_2, \cdots, \tau_d\}$ is a row covering of $\ve{Q}$ where each $\mathcal{G}_{\tau_i}$ is a complete subgraph of $\mathcal{G}$, we make use of Lemma~\ref{spanninglemma}; in particular, for each $\tau \in T$ there exists $\tau' \subset \tau$ such that $\mathcal{G}_{\tau'}$ is a spanning \emph{path} of $G_\tau$.  See Figures~\ref{fig:AC3} and \ref{fig:AC4} for an example of how a spanning path block update coincides with the update produced by a clique block update.

In this way, we see that the bound on the convergence rate constant for complete subgraphs must be no larger than that of path subgraphs and we recover the constant $$\left(1 - \frac{r\alpha(\mathcal{G})}{(2 - 2\cos \frac{M\pi}{M+1}) d}\right) \le \left(1 - \frac{r\alpha(\mathcal{G})}{4 d}\right).$$

\subsection{Arbitrary connected subgraph blocks}

In the case that $T = \{\tau_1, \tau_2, \cdots, \tau_d\}$ is a row covering of $\ve{Q}$ where each $\mathcal{G}_{\tau_i}$ is an arbitrary connected subgraph of $\mathcal{G}$, we again make use of the fact that we may replace every block $\tau \in T$ with a block $\tau' \subset \tau$ such that $\mathcal{G}_{\tau'}$ is a spanning tree of $\mathcal{G}_{\tau}$ to form $T' = \{\tau_1', \tau_2', \cdots, \tau_d'\}$.  We note that the block gossip method with blocks sampled from $T$ will produce the same set of iterates as those sampled identically from $T'$.

Now, we use that fact that the eigenvalues of the Laplacian of a tree, $L_\mathcal{T}$, are bounded above by $|\mathcal{E}(\mathcal{T})|$ (and that in fact this bound is tight for the star graph), see \cite{alggraphtheorybook}.  Thus, we have $$\lambda_{\max}(\ve{Q}_{\tau_i'}\ve{Q}_{\tau_i'}^\top) = \lambda_{\max}(\ve{Q}_{\tau_i'}^\top \ve{Q}_{\tau_i'}) \le |\tau_i| \le M.$$ We use this bound to recover the constant upper bound $$\left(1 - \frac{r\alpha(\mathcal{G})}{M d}\right).$$

\subsection{Multiple subgraph blocks}
%
If the network allows for multiple disjoint components to be activated at a single instance, we may form blocks consisting of multiple disjoint subgraphs of $\mathcal{G}$. In this case, to compute $\beta$ we may use the following Lemma.

\begin{lemma}
\label{laplacian_union}
Let $\mathcal{G}_\tau$ be a subgraph of $\mathcal{G}$ consisting of disjoint connected subgraphs $\mathcal{G}_{\tau_1}, \cdots, \mathcal{G}_{\tau_k}$. Then 
\begin{equation}
    \lambda_{\max}(\ve{Q}_\tau \ve{Q}_\tau^\top) = \max_i \lambda_{\max}(\ve{Q}_{\tau_i} \ve{Q}_{\tau_i}^\top).
\end{equation}
\end{lemma}
\begin{proof}
This follows immediately from the fact that since said subgraphs are edge-disjoint, we have that the Laplacian of their union is the direct sum of their individual Laplacians:
\begin{equation}
    \ve{L}_\tau = \ve{L}_{\tau_1} \oplus \cdots \oplus \ve{L}_{\tau_k},
\end{equation}
and so the spectrum of $\ve{L}_\tau$ is exactly the union of the spectra of $\ve{L}_{\tau_1}, \cdots, \ve{L}_{\tau_k}$.
\end{proof}

Suppose then that we take $T = \{\tau_1, \tau_2, \cdots, \tau_d\}$, where each $\tau_i$ is the union of a collection of rows corresponding to disjoint connected subgraphs, say $\tau_i = \bigcup_{j}\tau^{j}_i$. Then by Lemma~\ref{laplacian_union} we have 
\[
\lambda_{\max}(\ve{Q}_{\tau_i}\ve{Q}_{\tau_i}^\top) = \max_{j}\lambda_{\max}(\ve{Q}_{\tau^j_i}\ve{Q}_{\tau^j_i}^\top).
\]
One may then apply the relevant previous results of Section~\ref{sec:BGsampling} to compute an upper bound on this quantity, yielding $\beta$.




\section{Inconsistent Consensus Models}\label{sec:inconsistent}

In this section, we consider two models of inconsistent average consensus where communication across edges is noisy and provide analyses of the natural block gossip method in these cases. 

\subsection{Constant edge communication error}

Consider the average consensus problem in the presence of constant edge communication error; that is, some blocks of nodes do not update to local consensus during iterations of the block gossip method, but instead update according to an attempt to satisfy constant edge miscommunication values.  We denote $\ve{m} \in \mathbb{R}^{|\mathcal{E}|}$ as the edge miscommunication values and consider the block gossip updates under this edge miscommunication to be 
\begin{equation}\label{eq:noisyBG}
    \ve{c}_k = \ve{c}_{k-1} + \ve{Q}^\dagger_\tau(\ve{m}_\tau - \ve{Q}_\tau\ve{c}_{k-1}).
\end{equation}  We can apply Theorem~\ref{thm:main} to prove the following corollary.  This result yields a guarantee of convergence to a convergence horizon that depends upon the edge miscommunication vector $\ve{m}$.

\begin{corollary}\label{cor:constNoisyBG}
Suppose graph $\mathcal{G} = (\mathcal{V},\mathcal{E})$ is connected, $\ve{Q} \in \mathbb{R}^{|\mathcal{E}| \times |\mathcal{V}|}$ is the incidence matrix for $\mathcal{G}$, and $T = \{\tau_1, \cdots,\tau_d\}$ is a $(d,\alpha,\beta,r,R)$ row covering for $\ve{Q}$.  Then the block gossip method under edge miscommunication $\ve{m}$ as defined in \eqref{eq:noisyBG} with blocks determined by $T$ converges at least linearly in expectation to a horizon determined by $\ve{m}$ with the guarantee $$\mathbb{E}\norm{\ve{c}_k - \ve{c}^*}^2 \le \left(1 - \frac{r\alpha(\mathcal{G})}{\beta d}\right)^k \norm{\ve{c} - \ve{c}^*}^2 + \frac{\beta R}{\alpha r \alpha(\mathcal{G})}\|\ve{m}\|^2,$$ where $\alpha(\mathcal{G})$ is the algebraic connectivity of graph $\mathcal{G}$. 
\end{corollary}

\begin{proof}
This result follows from Theorem~\ref{thm:main} where $\ve{b} = \ve{m}$, $\ve{A} = \ve{Q}$, and $\ve{e} = -\ve{m}$.  Note that $\ve{x}^* = (\ve{I} - \ve{Q}^\dagger \ve{Q})\ve{c} + \ve{Q}^\dagger(\ve{m} - \ve{m}) = (\ve{I} - \ve{Q}^\dagger \ve{Q})\ve{c} = \ve{c}^*$.
\end{proof}

\subsection{Randomly varying edge communication error}

We now consider the average consensus problem in the presence of randomly varying edge communication error.  During the block gossip method, blocks do not update to local consensus but instead update to attempt to satisfy the iteration dependent edge miscommunication values.  We denote $\ve{m}_k \in \mathbb{R}^{|\mathcal{E}|}$ as the edge miscommunication values during the $k$th iteration and consider the block gossip updates under edge miscommunication to be 
\begin{equation}\label{eq:randomnoisyBG}
    \ve{c}_k = \ve{c}_{k-1} + \ve{Q}^\dagger_\tau((\ve{m}_k)_\tau - \ve{Q}_\tau\ve{c}_{k-1}).
\end{equation}
We prove a generalization of Theorem~\ref{thm:main} and use it to prove a guarantee of convergence to a convergence horizon that depends upon the distribution of the edge miscommunication values.

\begin{proposition}\label{prop:randomNoiseBRK}
Let $\ve{b} \in \text{r}(\ve{A})$ with $\ve{A}$ not necessarily full rank.  Consider running the block Kaczmarz method with matrix $\ve{A}$ and vector $\ve{b}_k = \ve{b} + \ve{e}_k$ in the $k$th iteration; that is $$\ve{x}_k = \ve{x}_{k-1} + \ve{A}_{\tau}^\dagger((\ve{b}_k)_\tau - \ve{A}_\tau\ve{x}_{k-1}).$$  Assume that $\{\ve{e}_k\}$ is sampled i.i.d.\ according to distribution $\mathscr{D}$, $\ve{e}_k \sim \mathscr{D}$, with $\mathbb{E}_{\mathscr{D}}[\ve{e}_k] = \ve{0}$ and $\text{cov}(\ve{e}_k) = \mathbb{E}_{\mathscr{D}}[\ve{e}_k \ve{e}_k^\top] = \ve{\Sigma}$.  Let $T = \{\tau_1, \tau_2, \cdots, \tau_d\}$ be a $(d,\alpha,\beta,r,R)$ row covering of $A$. Let $\ve{x}^* = (\ve{I} - \ve{A}^\dagger \ve{A})\ve{x}_0 + \ve{A}^\dagger\ve{b}$.  Then we have $$\mathbb{E} \|\ve{x}_{j+1} - \ve{x}^*\|^2 \le \left(1 - \frac{r\sigma_{\min +}^2(\ve{A})}{\beta d}\right)^{j+1}\|\ve{x}_0 - \ve{x}^*\|^2 + \frac{\beta R}{\alpha r \sigma_{\min +}^2(\ve{A})} \text{tr}(\ve{\Sigma}).$$
\end{proposition}

\begin{proof}
This proof proceeds in a manner highly similar to that of Theorem~\ref{thm:main}.  First, note that in a calculation similar to \eqref{eq:errorcalc}, we have that $\ve{A}\ve{x}^* - \ve{b}_k = \ve{e}_k$.  Then, recalling that the updates take the form $\ve{x}_{j+1} = \ve{x}_j + \ve{A}_\tau^\dagger((\ve{b}_k)_\tau - \ve{A}_\tau \ve{x}_j)$ where $\tau$ is chosen uniformly from $T$, we compute $$\|\ve{x}_{j+1} - \ve{x}^*\|^2 = \|(\ve{I} - \ve{A}_\tau^\dagger \ve{A}_\tau)(\ve{x}_j - \ve{x}^*)\|^2 + \|\ve{A}_\tau^\dagger(\ve{e}_j)_\tau\|^2$$ in a manner similar to that of \eqref{eq:errorsquared}. Using \eqref{eq:breakuperror} and taking expectation with respect to the sampled block in iteration $j$, $\tau_j$, conditioned on all previously sampled blocks, we have 
\begin{align*}
    \mathbb{E}_{\tau_j} \|\ve{x}_{j+1} - \ve{x}^*\|^2 &= \|\ve{x}_j - \ve{x}^*\|^2 - \mathbb{E}_j \|\ve{A}_\tau^\dagger \ve{A}_\tau (\ve{x}_j - \ve{x}^*)\|^2 + \frac{1}{\alpha} \mathbb{E}_j \|(\ve{e}_j)_\tau\|^2\\
    &\le \left(1 - \frac{r\sigma_{\min +}^2(\ve{A})}{\beta d}\right)\|\ve{x}_j - \ve{x}^*\|^2 + \frac{R}{\alpha d} \|\ve{e}_j\|^2.
\end{align*}

Now, taking expectation with respect to the sampled error $\ve{e}_j \sim \mathcal{D}$ conditioned upon all previously sampled errors, we arrive at 
\begin{align*}
    \mathbb{E}_{\ve{e}_j}\left[\mathbb{E}_{\tau_j} \|\ve{x}_{j+1} - \ve{x}^*\|^2\right] &\le \left(1 - \frac{r\sigma_{\min +}^2(\ve{A})}{\beta d}\right)\|\ve{x}_j - \ve{x}^*\|^2 + \frac{R}{\alpha d} \mathbb{E}_{\ve{e}_j}\|\ve{e}_j\|^2 \\
    &= \left(1 - \frac{r\sigma_{\min +}^2(\ve{A})}{\beta d}\right)\|\ve{x}_j - \ve{x}^*\|^2 + \frac{R}{\alpha d} \text{tr}(\ve{\Sigma}).
\end{align*}
Iterating this expectation and proceeding inductively as in \eqref{eq:induction}, we arrive at the desired result.
\end{proof}

We may now use this result to prove the following guarantee for convergence of block gossip methods in the presence of randomly varying edge miscommunication error.

\begin{corollary}\label{cor:randNoisyBG}
Suppose graph $\mathcal{G} = (\mathcal{V},\mathcal{E})$ is connected, $\ve{Q} \in \mathbb{R}^{|\mathcal{E}| \times |\mathcal{V}|}$ is the incidence matrix for $\mathcal{G}$, and $T = \{\tau_1, \cdots,\tau_d\}$ is a $(d,\alpha,\beta,r,R)$ row covering for $\ve{Q}$.  Assume that $\{\ve{m}_k\}$ is sampled i.i.d.\ according to distribution $\mathcal{D}$, $\ve{m}_k \sim \mathcal{D}$, with $\mathbb{E}_{\mathcal{D}}[\ve{m}_k] = 0$ and $\text{cov}(\ve{m}_k) = \mathbb{E}_{\mathcal{D}}[\ve{m}_k\ve{m}_k^\top] = \ve{\Sigma}$. Then the block gossip method under randomly varying edge miscommunication $\ve{m}_k$ as defined in \eqref{eq:randomnoisyBG} with blocks determined by $T$ converges at least linearly in expectation to a horizon determined by $\text{tr}(\ve{\Sigma})$ with the guarantee $$\mathbb{E}\norm{\ve{c}_k - \ve{c}^*}^2 \le \left(1 - \frac{r\alpha(\mathcal{G})}{\beta d}\right)^k \norm{\ve{c} - \ve{c}^*}^2 + \frac{\beta R}{\alpha r \alpha(\mathcal{G})}\text{tr}(\ve{\Sigma}),$$ where $\alpha(\mathcal{G})$ is the algebraic connectivity of graph $\mathcal{G}$. 
\end{corollary}

\begin{proof}
This result follows from Proposition~\ref{prop:randomNoiseBRK} where $\ve{b} = \ve{0}$, $\ve{e}_k = \ve{m}_k$, and $\ve{A} = \ve{Q}$.  Note that $\ve{x}^* = (\ve{I} - \ve{Q}^\dagger \ve{Q})\ve{c} = \ve{c}^*$.
\end{proof}

\begin{remark}
In the case that the blocks consist of single rows, then updates~\eqref{eq:noisyBG} and \eqref{eq:randomnoisyBG} correspond to nodes updating to satisfy a misspecified (nonhomogenous) equation in the AC system, which could model link communication failure.  However, for larger blocks, the interpretation of these updates break down and it is less clear that they model a natural gossip process, as the individual node value updates are produced by the collection of edge miscommunications.  We note, however, that such nonhomogenous systems of equations arise in practice elsewhere, e.g., in rank aggregration from pairwise comparisons via Massey's method~\cite{massey1997statistical} or Hodgerank~\cite{jiang2011statistical}.  
\end{remark}

\section{Experiments}\label{sec:experiments}

In this section we present empirical results from applying block gossip with various choices of block structure to AC problems on multiple graph structures, including Erd\"os-R\'enyi graphs of varying connectivity and square lattice graphs. All experiments were conducted in Python 3.8 with the NetworkX~\cite{hagberg2008exploring} package used to generate and work with graph structures. We also present results from applying said protocols to inconsistent consensus models as detailed in Section~\ref{sec:inconsistent}. 

\subsection{Preliminaries}
Recall that an Erd\"os-R\'enyi graph ER($n$,$p$) on $n$ vertices is formed by randomly including edges between each pair of nodes independently with probability $p$. We choose to experiment on such graphs as they are popular models for real-life networks and highlight the effects of varying connectivity (by varying $p$) on the convergence of the considered gossip protocols. We also run experiments on $n \times n$ square lattice graphs, another widely-studied network structure; see e.g., \cite{pathgossip}.

For all graphs we perform experiments with four block sampling protocols: independent edge sets (IES), cliques, paths, and randomly selected blocks of fixed size. These protocols and the graph structures underlying them are detailed in Subsection~\ref{subsec:mainresults} and Section~\ref{sec:BGsampling}.

To produce an IES cover we use a greedy algorithm that repeatedly finds the largest independent edge set, then removes it from the graph until there are no remaining edges. Similarly, a clique edge cover of the graph is generated by a greedy algorithm that repeatedly finds the largest clique, then removes it from the graph until there are no remaining edges. For path gossip, paths are formed by selecting a node uniformly at random, adding a randomly selected neighbour to it, and continuing to sequentially add neighbours until we have a path of the desired length $l$. Randomly selected blocks are sampled by selecting edges uniformly at random to form a block of a specified size. The blocks generated by these algorithms are then passed into our block gossip algorithm, which randomly samples a block from the list of blocks at each iteration.


We produce two types of plot: \emph{collapse plots}, which are a visualization of individual node values by iteration (as in~\cite{brooks2020model}), and \emph{error plots}, which show the error at each iteration, $\norm{\ve{c}_k - \ve{c}^\ast}$, and for some examples also display the predicted upper bound on convergence given by \cref{cor:BG}. 




\subsection{Erd\"os-R\'enyi Graphs}

We apply each of our block sampling protocols to Erd\"os-R\'enyi graphs ER$(n,p)$ with $p = 0.2, 0.4, 0.6, 0.8, 1$ and $n=200$. 

In \crefrange{fig:ER_0.2_all_protocols}{fig:ER_1_all_protocols} we compare the performance of each protocol (with path length $l=10$ for path gossip) across a range of $p$.  It appears that, in the ER case, connectivity does not have a substantial effect on the relative performance of our protocols, with IES gossip consistently being the strongest by some margin. This aligns with the fact that independent edge sets will have the greatest \emph{node} overlap, particularly when compared to cliques -- in the sense that a single node is likely to be in many more independent edge sets than cliques -- and so a larger amount of information is transferred per iteration.

\begin{figure}
\centering
\begin{minipage}{.5\textwidth}
  \centering
  \includegraphics[scale=0.35]{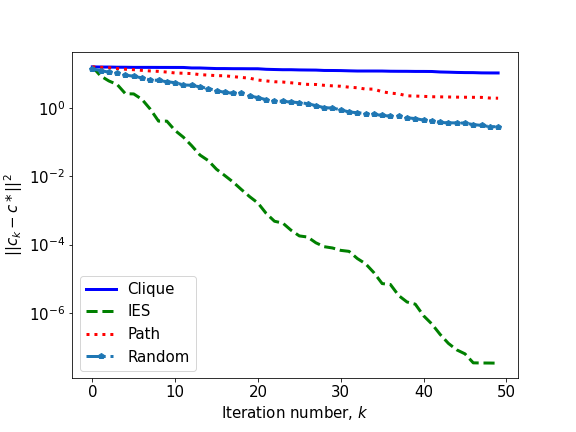}
  \caption{Errors for all protocols applied to ER$(200,0.2)$.}
  \label{fig:ER_0.2_all_protocols}
\end{minipage}%
\begin{minipage}{.5\textwidth}
  \centering
  \includegraphics[scale=0.35]{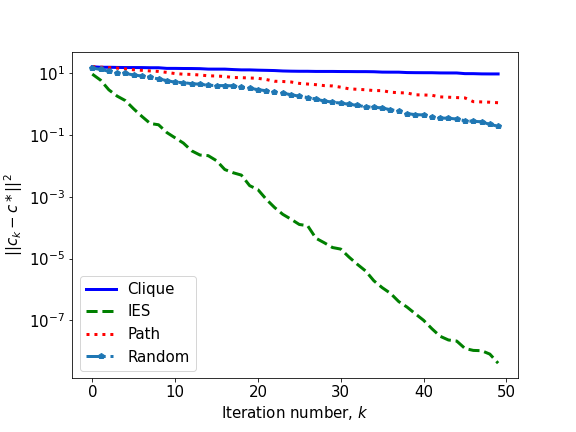}
  \caption{Errors for all protocols applied to ER$(200,0.4)$.}
  \label{fig:ER_0.4_all_protocols}
\end{minipage}
\end{figure}

\begin{figure}
\centering
\begin{minipage}{.5\textwidth}
  \centering
  \includegraphics[scale=0.35]{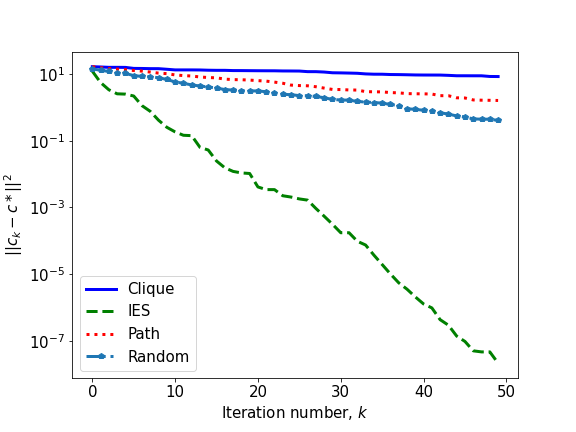}
  \caption{Errors for all protocols applied to ER$(200,0.6)$.}
  \label{fig:ER_0.6_all_protocols}
\end{minipage}%
\begin{minipage}{.5\textwidth}
  \centering
  \includegraphics[scale=0.35]{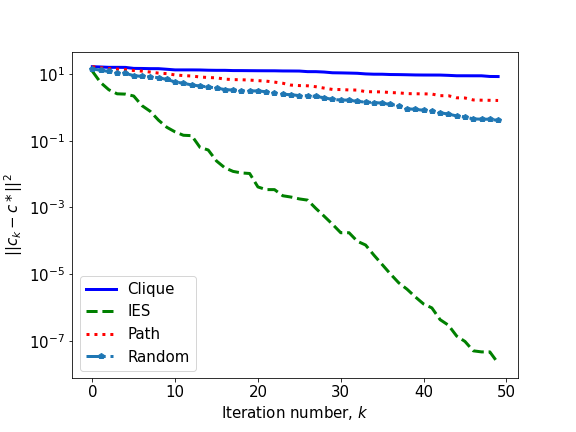}
  \caption{Errors for all protocols applied to ER$(200,1)$.}
  \label{fig:ER_1_all_protocols}
\end{minipage}
\end{figure}

\begin{figure}
\centering
\begin{minipage}{.5\textwidth}
  \centering
  \includegraphics[scale=0.35]{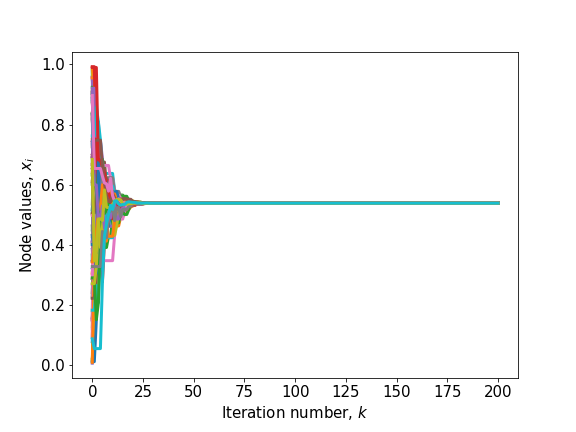}
  \caption{Collapse plot for ER$(100, 0.6)$ under IES gossip.}
  \label{fig:ER_0.6_ies_collapse}
\end{minipage}%
\begin{minipage}{.5\textwidth}
  \centering
  \includegraphics[scale=0.35]{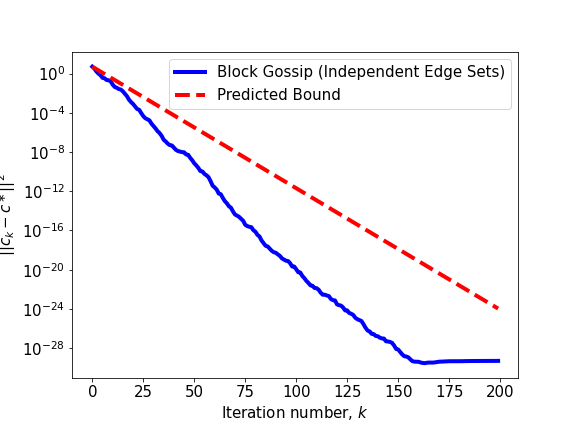}
  \caption{Error plot for ER$(100, 0.6)$ under IES gossip.}
  \label{fig:ER_0.6_ies_error}
\end{minipage}
\end{figure}

In \crefrange{fig:ER_0.6_ies_collapse}{fig:ER_0.6_clique_error}, we present a closer look at best (IES) and worst (clique) performing protocols from the previous experiments. The dramatic difference in convergence rate can be partially explained by the collapse plots: we see that during clique gossip certain nodes will hold their value for many iterations before updating, leading to dramatically slower convergence than IES gossip seen in the error plots. This is again connected to the greater amount of node overlap that persists in IES blocks, compared to cliques. We see from the error plots that both protocols respect the upper bound on convergence given by \cref{cor:BG}, and that said bound predicts that IES gossip should outperform clique gossip.


\begin{figure}
\centering
\begin{minipage}{.5\textwidth}
  \centering
  \includegraphics[scale=0.35]{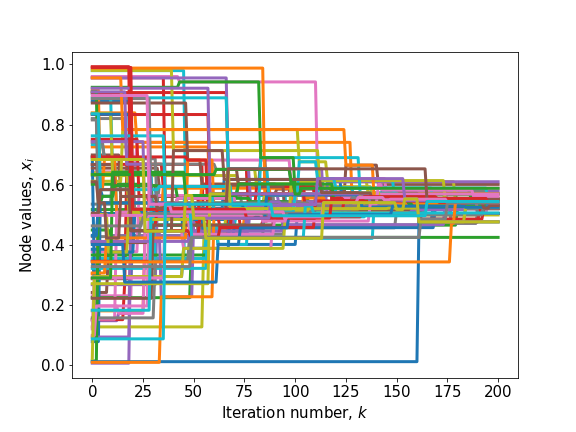}
  \caption{Collapse plot for ER$(100, 0.6)$ under clique gossip.}
  \label{fig:ER_0.6_clique_collapse}
\end{minipage}%
\begin{minipage}{.5\textwidth}
  \centering
  \includegraphics[scale=0.35]{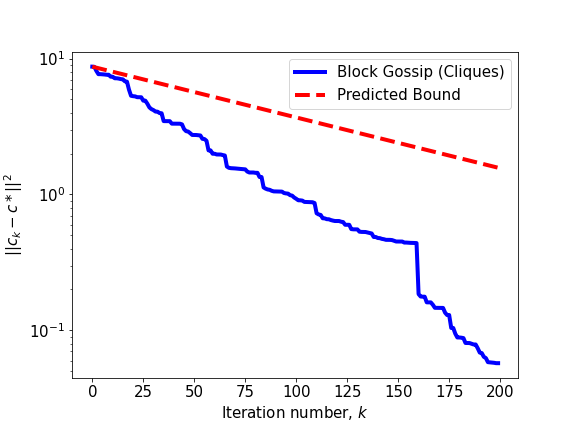}
  \caption{Error plot for ER$(100, 0.6)$ under clique gossip.}
  \label{fig:ER_0.6_clique_error}
\end{minipage}
\end{figure}

We analyze the effect of increasing $p$ (and thus the connectivity of the graph) in \crefrange{fig:varying_p_ies_errors}{fig:varying_p_random_errors}. It can be seen that IES gossip is both fast and robust to variations in connectivity compared to other protocols. This can be attributed heuristically to the fact that independent edge sets are formed by selecting edges which separate nodes well, rather than selecting edges which join nodes well (as in forming cliques). The performance of clique gossip improves significantly with $p$, corresponding to the fact that ER$(n,p)$ is likely to have a greater number of larger cliques as $p$, and thus its average degree, increases. Note that we exclude clique gossip for $p=1$, as the entire graph will be selected as a clique and consensus will be reached in a single iteration.

\begin{figure}
\centering
\begin{minipage}{.5\textwidth}
  \centering
  \includegraphics[scale=0.35]{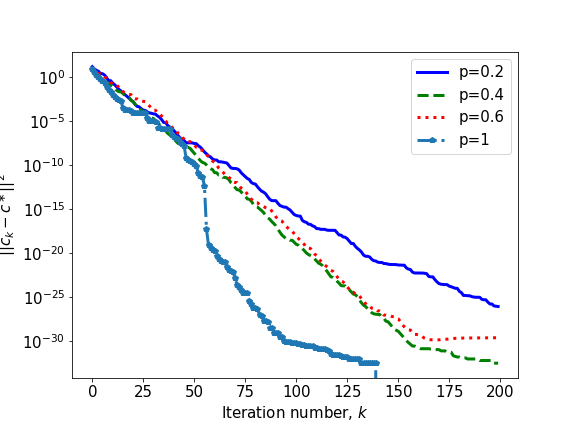}
  \caption{IES gossip errors on ER$(200,p)$ for various $p$.}
  \label{fig:varying_p_ies_errors}
\end{minipage}%
\begin{minipage}{.5\textwidth}
  \centering
  \includegraphics[scale=0.35]{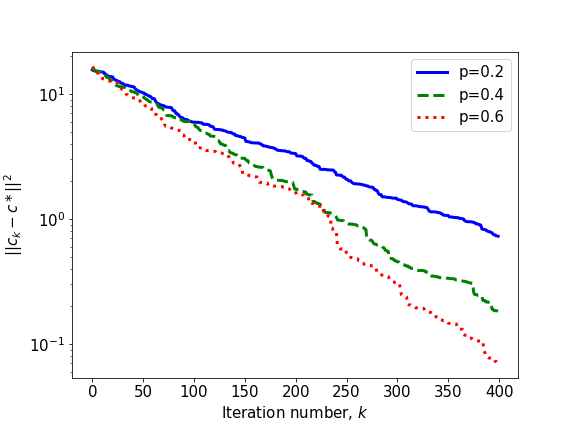}
  \caption{Clique gossip errors on ER$(200,p)$ for various $p$.}
  \label{fig:varying_p_clique_errors}
\end{minipage}
\end{figure}

\begin{figure}
\centering
\begin{minipage}{.5\textwidth}
  \centering
  \includegraphics[scale=0.35]{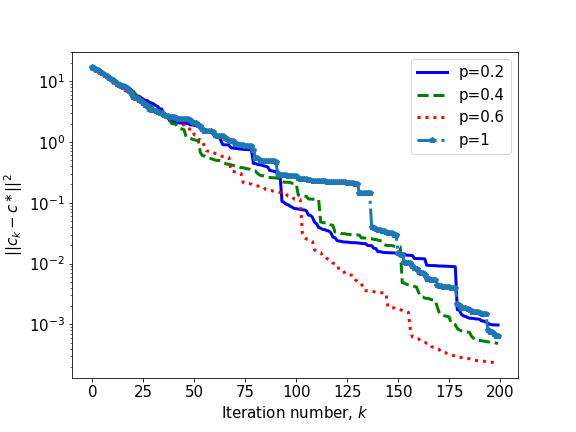}
  \caption{Path gossip errors on ER$(200,p)$ for various $p$.}
  \label{fig:varying_p_path_errors}
\end{minipage}%
\begin{minipage}{.5\textwidth}
  \centering
  \includegraphics[scale=0.35]{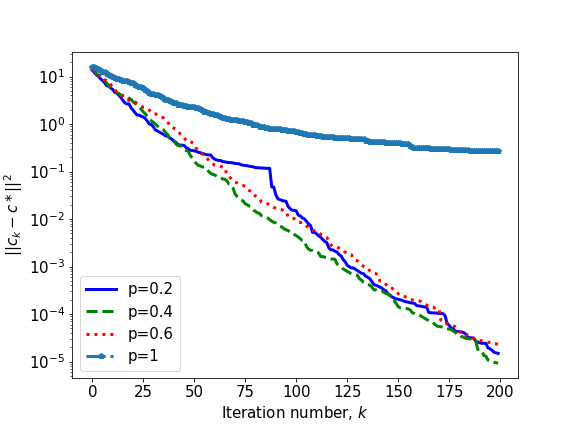}
  \caption{Random gossip errors on ER$(200,p)$ for various $p$.}
  \label{fig:varying_p_random_errors}
\end{minipage}
\end{figure}


\subsection{Square Lattice}

In \cref{fig:lattice_all_methods} we show the error plots from each of our sampling protocols applied to the $10 \times 10$ square lattice. The graph structure restricts the size of cliques to only single edges, leading to poor performance versus other protocols. Moreover, path gossip struggles as nodes on opposite sides of the grid are a large graph distance apart. In this scenario, with our `special' structures being limited, it is sensible to choose blocks at random to attempt to maximise the dispersement of information, and this strategy indeed yields the best performance.

\begin{figure}
    \centering
    \includegraphics[width=.5\textwidth]{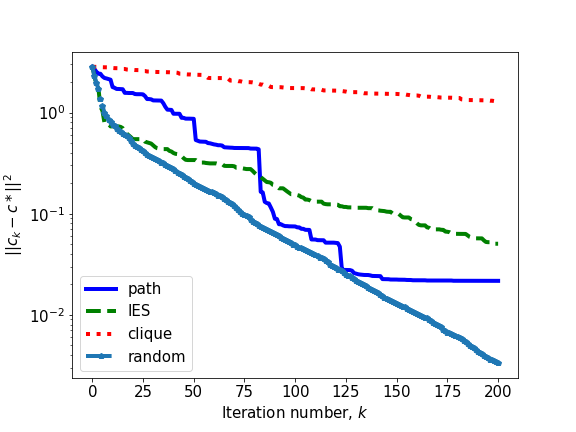}
    \caption{Comparison of sampling protocols applied to a $10 \times 10$ square lattice.}
    \label{fig:lattice_all_methods}
\end{figure}

\subsection{Inconsistent Consensus Models} We experimented with the two types of inconsistent average consensus systems described in Section~\ref{sec:inconsistent}, namely systems with a constant edge communication error (CECE) and systems with a randomly varying edge communication error (VECE). 
In CECE, all iterations are affected by the same edge miscommunication values $\ve{m} \in \mathbb{R}^{|\mathcal{E}|}$ which is generated (in advance) with entries sampled independently from $\mathcal{N}(0,0.01)$, the mean-zero Gaussian distribution with variance $0.01$. 
 In VECE, the edge miscommunication for the $k$th iteration, $\ve{m}_k \in \mathbb{R}^{|\mathcal{E}|}$, has entries sampled independently from $\mathcal{N}(0,0.01)$. 




In Figures~\ref{fig:path_com} and~\ref{fig:ies_com}, we display the errors for consistent block gossip updates ($\ve{m} = \ve{0}$), CECE updates~(\ref{eq:noisyBG}) with constant edge communication values $\ve{m}$ sampled as described above, and VECE updates~(\ref{eq:randomnoisyBG}) with varying edge communication values $\ve{m}_k$ sampled as described above, on an ER$(n,p)$ graph with $n = 150$ and $p = 0.6$ with path block sampling and IES block sampling, respectively.  Note that IES block sampling gossip converges far more quickly than path block sampling, and so we are able to see more clearly the effect of varying error in this plot.  Note that VECE has a fairly smooth convergence horizon, whereas the convergence horizon for CECE varies widely as the updates sample the (fixed) values of edge miscommunication of differing magnitude.

\begin{figure}
\centering
\begin{minipage}{.5\textwidth}
  \centering
  \includegraphics[scale=0.35]{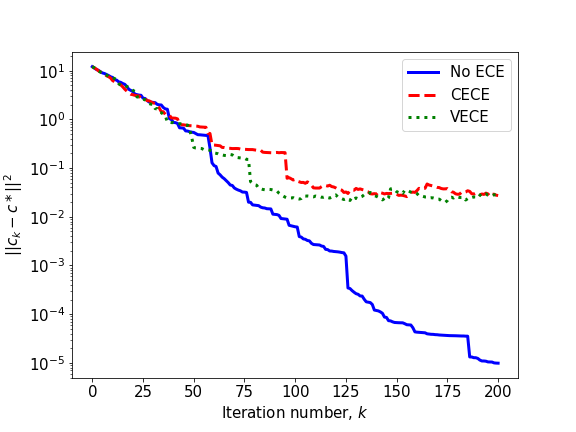}
  \caption{Comparison of effect of different communication errors on path gossip.}
  \label{fig:path_com}
\end{minipage}%
\begin{minipage}{.5\textwidth}
  \centering
  \includegraphics[scale=0.35]{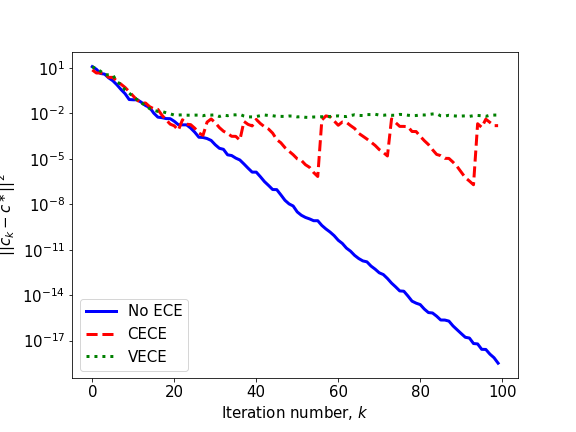}
  \caption{Comparison of effect of different communication errors on IES gossip.}
  \label{fig:ies_com}
\end{minipage}
\end{figure}

In Figures~\ref{fig:cli_com} and~\ref{fig:ran_com}, we display the errors for consistent block gossip updates ($\ve{m} = \ve{0}$), CECE updates~(\ref{eq:noisyBG}) with constant edge communication values $\ve{m}$ sampled as described above, and VECE updates~(\ref{eq:randomnoisyBG}) with varying edge communication values $\ve{m}_k$ sampled as described above, on an ER$(n,p)$ graph with $n = 150$ and $p = 0.6$ with clique block sampling and random block sampling, respectively.  Note that clique block gossip and random block gossip converge so slowly that it takes far longer to see clearly the convergence horizon for the CECE and VECE updates. 

\begin{figure}
\centering
\begin{minipage}{.5\textwidth}
  \centering
  \includegraphics[scale=0.35]{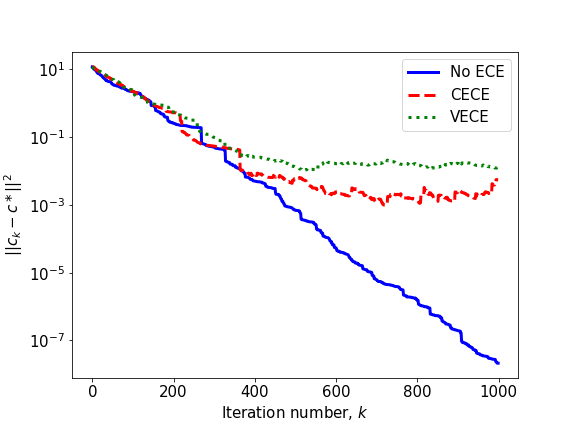}
  \caption{Comparison of effect of different communication errors on clique gossip.}
  \label{fig:cli_com}
\end{minipage}%
\begin{minipage}{.5\textwidth}
  \centering
  \includegraphics[scale=0.35]{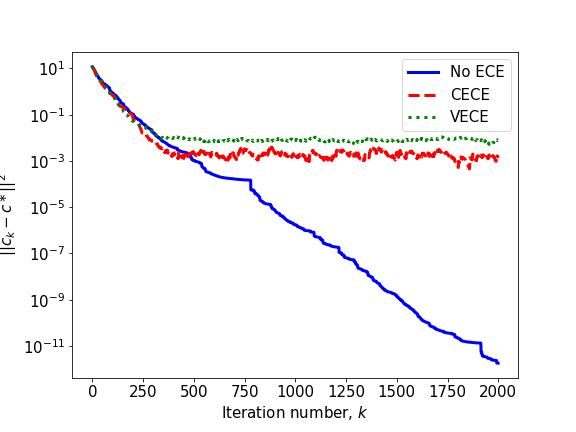}
  \caption{Comparison of effect of different communication errors on random gossip.}
  \label{fig:ran_com}
\end{minipage}
\end{figure}

\section{Conclusions}

In this work, we prove a new convergence result for the block gossip method for the average consensus problem and specialize this theoretical result to \emph{path}~\cite{pathgossip}, \emph{clique}~\cite{cliquegossip}, and \emph{edge-independent set}~\cite{Boyd2006RandomizedGA} gossip protocols.  
We prove this result by exploiting the fact that these methods are generalized by the block randomized Kaczmarz method for solving linear systems.  

We prove a generalized convergence result for the block randomized Kaczmarz method which generalizes the main result of~\cite{needell2013paved} to the case of rank-deficient systems and relaxes requirements on the set of blocks to be sampled.  While these generalization are highly important for the average consensus problem, we expect that they will be of interest in other applications as well.  

We additionally prove convergence results for the block gossip method on inconsistent consensus models, and perform a broad set of experiments to compare our theoretical results to the empirical behavior of the block gossip methods on various network structures.

Future directions include further exploration of inconsistent average consensus models, including random link failure and adversarial nodes, bounded confidence models, which generalize the average consensus model, and block Kaczmarz variants for randomly varying noise.

\section*{Acknowledgements}
The authors thank Jes\'us De Loera, Deanna Needell, Braxton Osting, and Mason Porter for useful conversations and suggestions.

\bibliographystyle{plain}
\bibliography{bib}

\begin{thebibliography}{10}

\bibitem{aysal2009broadcast}
T.~C. Aysal, M.~E. Yildiz, A.~D. Sarwate, and A.~Scaglione.
\newblock Broadcast gossip algorithms for consensus.
\newblock {\em IEEE T. Signal Proces.}, 57(7):2748--2761, 2009.

\bibitem{Boyd2006RandomizedGA}
S.~P. Boyd, A.~Ghosh, B.~Prabhakar, and D.~Shah.
\newblock Randomized gossip algorithms.
\newblock {\em IEEE T. Inform. Theory}, 52:2508--2530, 2006.

\bibitem{brooks2020model}
H.~Z. Brooks and M.~A. Porter.
\newblock A model for the influence of media on the ideology of content in
  online social networks.
\newblock {\em Phys. Rev. Res.}, 2(2):023041, 2020.

\bibitem{pathgossip}
F.~Bénézit, A.~G. Dimakis, P.~Thiran, and M.~Vetterli.
\newblock Order-optimal consensus through randomized path averaging.
\newblock {\em IEEE T. Inform. Theory}, 56(10):5150--5167, 2010.

\bibitem{cybenko1989dynamic}
G.~Cybenko.
\newblock Dynamic load balancing for distributed memory multiprocessors.
\newblock {\em J. Parallel Distr. Com.}, 7(2):279--301, 1989.

\bibitem{degroot1974}
M.~H. DeGroot.
\newblock Reaching a consensus.
\newblock {\em J. Am. Stat. Assoc.}, 69(345):118--121, 1974.

\bibitem{dimakis2010survey}
A.~G. Dimakis, S.~Kar, J.~M.~F. Moura, M.~G. Rabbat, and A.~Scaglione.
\newblock Gossip algorithms for distributed signal processing.
\newblock {\em P. IEEE}, 98(11):1847--1864, 2010.

\bibitem{dimakis2008geographic}
A.~G. Dimakis, A.~D. Sarwate, and M.~J. Wainwright.
\newblock Geographic gossip: Efficient averaging for sensor networks.
\newblock {\em IEEE T. Signal Proces.}, 56(3):1205--1216, 2008.

\bibitem{Eggermont1981IterativeAF}
P.~Eggermont, G.~Herman, and A.~Lent.
\newblock Iterative algorithms for large partitioned linear systems, with
  applications to image reconstruction.
\newblock {\em Linear Algebra Appl.}, 40:37--67, 1981.

\bibitem{Elf80:Block-Iterative-Methods}
T.~Elfving.
\newblock Block-iterative methods for consistent and inconsistent linear
  equations.
\newblock {\em Numer. Math.}, 35(1):1--12, 1980.

\bibitem{freris2012fast}
N.~M. Freris and A.~Zouzias.
\newblock Fast distributed smoothing of relative measurements.
\newblock In {\em Proc. 51st IEEE Conf. on Decision and Control (CDC)}, pages
  1411--1416. IEEE, 2012.

\bibitem{alggraphtheorybook}
C.~Godsil and G.~Royle.
\newblock {\em Algebraic Graph Theory}, volume 207.
\newblock 01 2001.

\bibitem{HM19Greed}
J.~Haddock and A.~Ma.
\newblock Greed works: An improved analysis of sampling {K}aczmarz-{M}otzkin.
\newblock {\em SIAM J. Math. Data Sci.}, 2020.
\newblock to appear.

\bibitem{hagberg2008exploring}
A.~Hagberg, P.~Swart, and D.~S.~Chult.
\newblock Exploring network structure, dynamics, and function using
  {N}etwork{X}.
\newblock Technical report, Los Alamos National Lab. (LANL), Los Alamos, NM
  (United States), 2008.

\bibitem{hanzely2017privacy}
F.~Hanzely, J.~Kone{\v{c}}n{\`y}, N.~Loizou, P.~Richt{\'a}rik, and
  D.~Grishchenko.
\newblock Privacy preserving randomized gossip algorithms.
\newblock {\em arXiv preprint arXiv:1706.07636}, 2017.

\bibitem{jiang2011statistical}
X.~Jiang, L.-H. Lim, Y.~Yao, and Y.~Ye.
\newblock Statistical ranking and combinatorial {H}odge theory.
\newblock {\em Math. Program.}, 127(1):203--244, 2011.

\bibitem{ogkacz}
S.~Kaczmarz.
\newblock Angenäherte auflösung von systemen linearer gleichungen.
\newblock {\em Bull. Internat. Acad. Polon. Sci. Lettres A}, page 335–357,
  1937.

\bibitem{kempe2003gossip}
D.~Kempe, A.~Dobra, and J.~Gehrke.
\newblock Gossip-based computation of aggregate information.
\newblock In {\em Ann. IEEE Symp. Found.}, pages 482--491. IEEE, 2003.

\bibitem{liu2013analysis}
J.~Liu, B.~D.O. Anderson, M.~Cao, and A.~S. Morse.
\newblock Analysis of accelerated gossip algorithms.
\newblock {\em Automatika}, 49(4):873--883, 2013.

\bibitem{cliquegossip}
Y.~Liu, B.~Li, B.~O. Anderson, and G.~Shi.
\newblock Clique gossiping.
\newblock {\em IEEE/ACM Transactions on Networking}, 27(06):2418--2431, nov
  2019.

\bibitem{loizou2016new}
N.~Loizou and P.~Richt{\'a}rik.
\newblock A new perspective on randomized gossip algorithms.
\newblock In {\em IEEE Glob. Conf. Sig.}, pages 440--444. IEEE, 2016.

\bibitem{loizou2019revisiting}
N.~Loizou and P.~Richt{\'a}rik.
\newblock Revisiting randomized gossip algorithms: General framework,
  convergence rates and novel block and accelerated protocols.
\newblock {\em arXiv preprint arXiv:1905.08645}, 2019.

\bibitem{massey1997statistical}
K.~Massey.
\newblock Statistical models applied to the rating of sports teams.
\newblock {\em Bluefield College}, 1997.

\bibitem{necoara2019faster}
I.~Necoara.
\newblock Faster randomized block {K}aczmarz algorithms.
\newblock {\em SIAM J. Matrix Anal. A.}, 40(4):1425--1452, 2019.

\bibitem{nedic2018network}
A.~Nedi{\'c}, A.~Olshevsky, and M.~G. Rabbat.
\newblock Network topology and communication-computation tradeoffs in
  decentralized optimization.
\newblock {\em P. IEEE}, 106(5):953--976, 2018.

\bibitem{Nee10:Randomized-Kaczmarz}
D.~Needell.
\newblock Randomized {K}aczmarz solver for noisy linear systems.
\newblock {\em BIT}, 50(2):395--403, 2010.

\bibitem{needell2013paved}
D.~Needell and J.~A. Tropp.
\newblock Paved with good intentions: Analysis of a randomized block {K}aczmarz
  method.
\newblock {\em Linear Algebra Appl.}, 2013.

\bibitem{needell2015randomized}
D.~Needell, R.~Zhao, and A.~Zouzias.
\newblock Randomized block {K}aczmarz method with projection for solving least
  squares.
\newblock {\em Linear Algebra Appl.}, 484:322--343, 2015.

\bibitem{sabergossip2007}
R.~Olfati-Saber, J.~A. Fax, and R.~M. Murray.
\newblock Consensus and cooperation in networked multi-agent systems.
\newblock {\em P. IEEE}, 95(1):215--233, 2007.

\bibitem{olshevsky2009convergence}
A.~Olshevsky and J.~N. Tsitsiklis.
\newblock Convergence speed in distributed consensus and averaging.
\newblock {\em SIAM J. Control Optim.}, 48(1):33--55, 2009.

\bibitem{Pop99:Block-Projections-Algorithms}
C.~Popa.
\newblock Block-projections algorithms with blocks containing mutually
  orthogonal rows and columns.
\newblock {\em BIT}, 39(2):323--338, 1999.

\bibitem{vershstrohkacz}
T.~Strohmer and R.~Vershynin.
\newblock A randomized {K}aczmarz algorithm with exponential convergence.
\newblock {\em J. Fourier Anal. Appl.}, 15:262--278, 2007.

\bibitem{tsitsiklis1986}
J.~Tsitsiklis, D.~Bertsekas, and M.~Athans.
\newblock Distributed asynchronous deterministic and stochastic gradient
  optimization algorithms.
\newblock {\em IEEE T. Automat. Contr.}, 31(9):803--812, 1986.

\bibitem{weberpaving}
E.~Weber.
\newblock Algebraic aspects of the paving and {F}eichtinger conjectures.
\newblock In {\em Topics in Operator Theory}, pages 569--578. Birkh{\"a}user
  Basel, 2010.

\bibitem{xiao2007distributed}
L.~Xiao, S.~Boyd, and S.-J. Kim.
\newblock Distributed average consensus with least-mean-square deviation.
\newblock {\em J. Parallel Distr. Com.}, 67(1):33--46, 2007.

\bibitem{xiao2005scheme}
L.~Xiao, S.~Boyd, and S.~Lall.
\newblock A scheme for robust distributed sensor fusion based on average
  consensus.
\newblock In {\em Proc. Int. Symp. on Information Processing in Sensor
  Networks}, pages 63--70. IEEE, 2005.

\bibitem{zhang2019distributed}
J.~Zhang, J.~Cui, Z.~Wang, Y.~Ding, and Y.~Xia.
\newblock Distributed joint cooperative self-localization and target tracking
  algorithm for mobile networks.
\newblock {\em Sensors}, 19(18):3829, 2019.

\bibitem{rek}
A.~Zouzias and N.~M. Freris.
\newblock Randomized extended {K}aczmarz for solving least squares.
\newblock {\em SIAM J. Matrix Anal. A.}, 34(2):773–793, 2013.

\bibitem{zouzias2015randomized}
A.~Zouzias and N.~M. Freris.
\newblock Randomized gossip algorithms for solving {L}aplacian systems.
\newblock In {\em Proc. European Control Conference (ECC)}, pages 1920--1925.
  IEEE, 2015.

\end{thebibliography}

\end{document}